\documentclass{amsart}

\usepackage{amssymb,latexsym,amscd}

\numberwithin{equation}{section}
\renewcommand{\thesection}{\arabic{section}}

\long\def\eatit#1{}

\newtheorem{thm}{Theorem}[subsection]
\newtheorem{prop}[thm]{Proposition}
\newtheorem{lem}[thm]{Lemma}
\newtheorem{cor}[thm]{Corollary}

\theoremstyle{definition}

\newtheorem{Ex}[thm]{Example}
\newtheorem{Rmk}[thm]{Remark}

\begin{document}


\title{Comparing Powers and Symbolic Powers of Ideals}

\author{Cristiano Bocci \& Brian Harbourne}

\address{Cristiano Bocci\\ 
Dipartimento di Scienze Matematiche e Informatiche "R. Magari"\\
Universit\`a degli Studi di Siena\\
Pian dei mantellini, 44\\
53100 Siena, Italy}
\email{bocci24@unisi.it}

\address{Brian Harbourne\\
Department of Mathematics\\
University of Nebraska\\
Lincoln, NE 68588-0130 USA}
\email{bharbour@math.unl.edu}

\date{June 19, 2009}

\thanks{This reseach was partially supported by GNSAGA of INdAM (Italy)
and by the NSA. We thank L. Avramov, M. Chardin, L. Chiantini,
L. Ein, C. Huneke, S. Iyengar, 
D. Katz, J. Migliore, T. Marley and Z. Teitler for helpful comments.}

\begin{abstract} We develop tools to study the problem of containment
of symbolic powers $I^{(m)}$ in powers $I^r$ for a homogeneous
ideal $I$ in a polynomial ring 
$k[{\bf P}^N]$ in $N+1$ variables over an arbitrary algebraically closed 
field $k$. We obtain results on the structure of the set of pairs 
$(r,m)$ such that $I^{(m)}\subseteq I^r$. As corollaries,
we show that $I^2$ contains $I^{(3)}$
whenever $S$ is a finite generic set of points in ${\bf P}^2$
(thereby giving a partial answer to a question of Huneke), and
we show that the containment theorems of \cite{refELS} and \cite{refHH1}
are optimal for every fixed dimension and codimension. 
\end{abstract}


\maketitle



\section{Introduction}\label{intro}

Consider a homogeneous ideal $I$ in a polynomial
ring $k[{\bf P}^N]$. Taking powers of $I$ is a natural 
algebraic construction, but it can be
difficult to understand their structure
geometrically (for example, knowing generators
of $I^r$ does not make it easy to know its
primary decomposition). On the other hand,
symbolic powers of $I$ are more natural geometrically
than algebraically. For example, if $I$ is a radical ideal
defining a finite set of points $p_1,\ldots,p_s\in{\bf P}^N$, then
its $m$th symbolic power $I^{(m)}$ is generated by all
forms vanishing to order at least $m$ at each point $p_i$,
but it is not easy to write down specific generators
for $I^{(m)}$, even if one has generators for $I$.

Thus it is of interest to compare the two constructions,
and a good deal of work has been done recently comparing 
powers of ideals with symbolic powers in various ways.
See for example, \cite{refHo}, \cite{refS}, \cite{refK}, \cite{refELS}, \cite{refHH1},
\cite{refCHHT} and \cite{refLS}.
Here we ask when a power of $I$ contains
a symbolic power, or vice versa. The second question
has an easy answer: if $I$ is nontrivial (i.e., not $(0)$ or $(1)$),
then $I^r\subseteq I^{(m)}$ if and only if $m\le r$ \cite[Lemma 8.1.4]{refPSC}.
Thus here we focus on the first question, 
and for that question all that it is easy to say 
is that if $I$ is nontrivial and $I^{(m)}\subseteq I^r$, then $m\ge r$ (Lemma \ref{GGPlem}(a)).
The problem of precisely for which $m\ge r$ we have 
$I^{(m)}\subseteq I^r$ is largely open. 

As a stepping stone, we introduce an asymptotic
quantity which we refer to as 
the {\it resurgence}, namely 
$\rho(I)=\hbox{sup}\{m/r : I^{(m)}\not\subseteq I^r\}$.
In particular, if $m > \rho(I)r$, then one is guaranteed
that $I^{(m)}\subseteq I^r$. Until recently it would not have been clear
that the sup always exists, but results of \cite{refS}
imply, for radical ideals at least, that it does, 
and \cite{refHH1}, generalizing the result of \cite{refELS},
shows in fact that $\rho(I)\le N$
and hence for a nontrivial homogeneous ideal $I$ we have $1\le \rho(I)\le N$ (see Lemma \ref{postcrit1}(b)).

There are still, however, very few cases for which the actual value of
$\rho(I)$ is known, and they are almost all cases for which
$\rho(I)=1$. For example, by Macaulay's unmixedness theorem
it follows that $\rho(I)=1$ when $I$ is a complete intersection
(also see \cite{refHo} and \cite{refLS}). And if $I$ is a monomial ideal, 
it is sometimes possible to compute $\rho(I)$ directly; for example,
if $I$ defines three noncollinear points in ${\bf P}^2$, then
one can show $\rho(I)=4/3$ (see \cite{refBH2}).

In this paper we give the first results
regarding the structure of the set of pairs
$(r,m)$ for which $I^{(m)}\subseteq I^r$.
These results are in terms of numerical invariants
of $I$. In particular, let $\alpha(I)$ be the 
least degree of a generator in any set of homogeneous generators of $I$,
let $\omega(I)$ be the least degree $t$ such that $I$
is generated by forms of degree $t$ and less,
and let ${\rm reg}(I)$ be the regularity of $I$.
We also define an invariant $\gamma(I)$,
which is like a Seshadri constant.
We then obtain the following structural results.
If $m/r\le \alpha(I)/\gamma(I)$, we prove that
$I^{(mt)}\not\subseteq I^{rt}$ for all $t\gg0$ (Lemma \ref{postcrit1}).
If in addition $I$ defines a zero dimensional subscheme,
then we show $m/r\ge {\rm reg}(I)/\gamma(I)$ implies that
$I^{(m)}\subseteq I^r$ (Corollary \ref{PCcor}), and we show that
$m/r > \omega(I)/\gamma(I)$ implies that
$I^{(mt)}\subseteq I^{rt}$ for all $t\gg0$ (Corollary \ref{asympcor}).
From these results it follows that $\alpha(I)/\gamma(I)\le \rho(I)$, and,
when $I$ defines a zero-dimensional subscheme of ${\bf P}^N$, 
that $\rho(I)\le {\rm reg}(I)/\gamma(I)$ (see Theorem \ref{SCthm}). 

By applying these results we give the first
determinations of $\rho(I)$ in cases for which $\rho(I)>1$
and $I$ is not monomial (see Theorem \ref{skeletonThm}(a) 
and Proposition \ref{coneprop}(a)). 
As a corollary, it follows that the upper bounds
on $\rho$ coming from \cite{refELS} and \cite{refHH1} are sharp
(see Corollary \ref{LESHHmaxThm}).

Our original motivation for this work was a question 
of Huneke's which is still open: if $I=I(S)$ is the ideal defining
any finite set $S$ of points in ${\bf P}^2$, is it true that $I^{(3)}\subseteq I^2$? This 
question was prompted by the results of \cite{refHH1} and \cite{refELS},
which guarantee that $I^{(4)}\subseteq I^2$. The question of 
the containment $I^{(3)}\subseteq I^2$ turns out to be quite delicate.
Here we show that containment holds at least when $S$ is a set
of generic points (Theorem \ref{8}).

\subsection{Comparison Invariants}

As mentioned above, given any homogeneous ideal $0\ne I\subsetneq R=k[{\bf P}^N]$,
we define the {\it resurgence}, $\rho(I)$, of $I$ to be the supremum of all ratios $m/r$ such  that
$I^r$ does not contain $I^{(m)}$, where by
$I^{(m)}$ we mean, as in \cite{refHH1}, the contraction of $I^mR_A$ to $R$, 
where $R_A$ is the localization of $R$ by the multiplicative system $A$, and $A$ is
the complement of the union of the 
associated primes of $I$.
We refer to the maximum height among the associated primes of $I$
as the codimension, $\hbox{cod}(I)$, of $I$.

The saturation $\hbox{sat}(I)$ of a homogeneous ideal  $I$ is 
the ideal generated by all forms $F$ such that
$(x_0,\ldots,x_N)^tF\subseteq I$ for some $t$ sufficiently large. 
If $I=\hbox{sat}(I)$, we say
$I$ is saturated. In any case, there is always a $t$ such that 
$I_j=\hbox{sat}(I)_j$ for all $j\ge t$. The least such $t$ is 
the saturation degree, $\hbox{satdeg}(I)$, of $I$.

In case $I$ is saturated and thus we have $I=I(X)$ for a subscheme
$X\subseteq {\bf P}^N$, we may write $\rho(X)$ to mean
$\rho(I)$.
A case of particular interest to us here is when $I(X)$ 
is an intersection $I=I(X)=\cap_i I(L_i)^{m_i}$
of powers of ideals of linear subspaces $L_i\subseteq{\bf P}^N$,
none of which contains another, in which case we refer to $X$ 
as a {\it fat flat} subscheme. Taking symbolic powers of $I(X)$ is then
straightforward; since $I(L_i)^{m_i}$ is primary, $I^{(m)}=\cap_i I(L_i)^{mm_i}$.
A special case of particular importance is when each 
space $L_i$ is a single reduced point $p_i$. In this case $X$ is known as
a {\it fat point subscheme} and $I^{(m)}$ is just 
the saturation of $I^m$.

Now let $\rho(N,d)$ denote the supremum of $\rho(I)$ over 
homogeneous ideals $0\ne I\subsetneq k[{\bf P}^N]$ of codimension $d$. 
(Since $d=N$ is an important special
case, will just write $\rho(N)$ for $\rho(N,N)$.) The main theorem of
\cite{refHH1} implies that $\rho(N,d)\le d$.
Here we show that in fact $\rho(N,d)=d$.

Although, as far as we know, it has not previously been shown for any $N$ or $d>1$
that $\rho(N,d)=d$, examples of Ein (see Section \ref{Optsubsect} and \cite{refHH2}) 
show that $\hbox{lim}_{N\to\infty}\rho(N,d)=d$.
We obtain:

\begin{cor}\label{LESHHmaxThm}
For each $N\ge 1$ and $1\le d\le N$, we have $\rho(N,d)=d$.
\end{cor}

Our proof of Corollary \ref{LESHHmaxThm} involves finding, for each $N$ and $d$,
a sequence of subschemes $S_N(d,i)\subsetneq {\bf P}^N$ such that 
$\lim_{i\to \infty}\rho(S_N(d,i))=d$. These subschemes 
can be taken to be fat flat subschemes, and, in fact, reduced.

Our main technical tool involves developing bounds, as discussed above, on $\rho(Z)$
for subchemes $Z\subsetneq {\bf P}^N$,
mostly in terms of postulational invariants of $I(Z)$; i.e., invariants
that are determined by the Hilbert functions of $I(Z)^{(m)}$.
Thus these bounds are the same
for any $Z$ for which the Hilbert functions of $I(Z)$ and its symbolic powers
remain the same. This is useful since postulational data is reasonably accessible,
either computationally or theoretically
(for example, \cite{refGuardoHar} and \cite{refGHM} classify all sets of up to 8 points
in ${\bf P}^2$ according to the postulational data of fat point subschemes
supported at the points).

\subsection{Postulational Bounds and Seshadri Constants}

We now discuss in detail the postulational invariants we will use.
Given a homogeneous ideal $0\ne I\subseteq R=k[{\bf P}^N]$, let 
$\alpha(I)$ be the least degree $t$ such that the homogeneous
component $I_t$ in degree $t$ is not zero. Thus $\alpha$ is, so to speak, the degree
in which the ideal begins. It is also the degree of a generator of least degree,
and it is the $M$-adic order of $I$ (i.e., the largest $t$ such that
$I\subseteq M^t$), where $M$ is the maximal homogeneous ideal.
If $Z\subseteq {\bf P}^{N-1}\subsetneq {\bf P}^N$ is a subscheme
contained in a hyperplane, in cases which are not clear from context
we will use 
$\alpha_{N-1}(I(Z))$ or $\alpha_N(I(Z))$ to distinguish
whether we are considering $\alpha$ for the ideal defining $Z$ in
${\bf P}^{N-1}$ or in ${\bf P}^N$.
Let $\tau(I)$ be the least degree such that
the Hilbert function  becomes equal to the Hilbert polynomial of $I$
and let $\sigma(I)=\tau(I)+1$.

Given a minimal free resolution $0\to F_N\to\cdots\to F_0\to I\to0$
of $I$ over $R$, where $F_i$ as a graded $R$-module is
$\oplus R[-b_{ij}]$,  the {\it Castelnuovo-Mumford regularity} ${\rm reg}(I)$
of $I$ is the maximum over all $i$ and $j$ of $b_{ij}-i$.
If $I$ defines a 0-dimensional subscheme of ${\bf P}^N$
(i.e., $I$ has codimension $N$),
then ${\rm reg}(I)$ is the maximum
of $\hbox{satdeg}(I)$ and $\sigma(\hbox{sat}(I))$, hence if $I$ is already
saturated (and so is the ideal of a 0-dimensional subscheme), then 
${\rm reg}(I)=\sigma(I)$ (see \cite{refGGP}).
(We will only be concerned with the regularity in case $I$ defines a 
0-dimensional subscheme.)

Our results depend on our developing bounds on $\rho(I)$.
Our bounds involve the quantity
$\gamma(I)=\hbox{lim}_{m\to \infty}\alpha(I^{(m)})/m$
for a homogeneous ideal $0\ne I\subsetneq k[{\bf P}^N]$.
Because of the subadditivity of $\alpha$,
this limit exists (see Remark III.7 of \cite{refHR2} or Lemma \ref{subadd}). 
Moreover, $\gamma(I)>0$ (see Lemma \ref{postcrit1}).
Given a subscheme $Z\subsetneq {\bf P}^N$, 
we will write $\gamma(Z)$ for $\gamma(I(Z))$.
Since $\alpha(I^m)$ is linear in $m$,
note that $\alpha(I)/\gamma(I)=\hbox{lim}_{m\to \infty} \alpha(I^m)/\alpha(I^{(m)})$.
Thus $\alpha(I)/\gamma(I)$ gives an asymptotic measure of the growth
of $I^{(m)}$ compared to $I^m$. 

Our next result thus shows that $\rho(I)$ measures additional growth,
in comparison to $\alpha(I)/\gamma(I)$
(hence the term resurgence for $\rho$).

\begin{thm}\label{SCthm}
Let $0\ne I\subsetneq k[{\bf P}^N]$ be a homogeneous ideal.
\begin{itemize}
\item[(a)] Then ${\alpha(I)/\gamma(I)}\le\rho(I)$.
\item[(b)] If in addition $I$ defines a 0-dimensional subscheme,
then $\rho(I)\le {\rm reg}(I)/\gamma(I)$.
\end{itemize}
\end{thm}

Thus, for example, given $I=I(Z)$ for a fat point subscheme $Z$ with $\alpha(I)=\sigma(I)$, this theorem shows that
computing $\rho(Z)$ is equivalent to computing $\gamma(Z)$.
The quantity $\gamma$ is in that case essentially  a uniform version of a multi-point Seshadri constant.
Indeed, if $Z$ is a reduced finite generic set of $n$ points in ${\bf P}^N$, then $\gamma(Z)=
n(\varepsilon(N,Z))^{N-1}$ (see Lemma \ref{subadd}), where,
following the exposition of \cite{refHR2, refHR3}, $\varepsilon(N,Z)$ is
the codimension 1 multipoint
Seshadri constant for $Z=\{p_1, \dots, p_n\}$; i.e., the real number
$$\varepsilon(N,Z)=\root{N-1}\of {\hbox{inf}
\left\{\frac{\hbox{deg}(H)}{\Sigma_{i=1}^n \hbox{mult}_{p_i}H}\right\}},$$ 
where the infimum is taken with respect to 
all hypersurfaces $H$, through at least one of the points (see \cite{refD}
and \cite{refXb}). We also define
$\varepsilon(N,n)$ to be ${\hbox{sup}\{\varepsilon(N,Z)\}},$ where
the supremum is taken with respect to 
all choices $Z$ consisting of $n$ distinct points $p_i$ of ${\bf P}^{N}$. 
In case $N$ is clear from context, we will write $\varepsilon(Z)$
for $\varepsilon(N,Z)$.

While it is in any case obvious from the definitions that $\gamma(Z)\ge
n(\varepsilon(N,Z))^{N-1}$, equality can fail since the latter takes notice of
hypersurfaces whose multiplicities at the points $p_i$ 
need not all be the same. (For example, if $Z$ is the reduced scheme
consisting of $n=4$ points in ${\bf P}^2$, 3 of them on a line and one off,
then  $5/3=\gamma(Z) >n\varepsilon(2,Z)=4/3$.) 

\subsection{Application to generic points}

As an interesting example, consider
${\bf P}^N$ and some $s$, and let $I$ be the ideal of
$n=\binom{s+N}{N}$ generic points of ${\bf P}^N$; then in
Theorem \ref{SCthm} we have $\alpha(I)=s+1=\sigma(I)={\rm reg}(I)$.
Although, in the case of $N=2$, $\varepsilon(2,n)$ (and hence $\gamma(I)$) is known for $n<10$,
a famous and still open conjecture of Nagata \cite{refNag} is equivalent 
to asserting that $\varepsilon(2,n)=1/\sqrt{n}$ for $n\ge 10$.
For no nonsquare $n\ge10$ is $\varepsilon(2,n)$ currently known.
However, it is not hard to show that $\varepsilon(2,n)=1/\sqrt{n}$ if 
$n$ is any square. Thus we have the following corollary.

\begin{cor}\label{SCcor}
If $n=\binom{s+N}{N}$, then
for the subscheme $Z\subset {\bf P}^N$ consisting of the union of
$n$ distinct generic points we have $\rho(Z)=\frac{s+1}{n(\varepsilon(N,Z))^{N-1}}$.
If in addition $N=2$ and $n$ is a square, then
$$\rho(Z)={\frac{s+1}{\sqrt{n}}}=\sqrt{2}\sqrt{\frac{s+1}{s+2}}.$$
\end{cor}

We remark that there are infinitely many integers $n$ which are 
at the same time a square and of the form $\binom{s+2}{2}$.
(An easy argument shows that $n=\binom{s+2}{2}$ is a square if and only if
either $s+1=2x^2$ for some $y$ such that $y^2-2x^2=1$,
or $s+2=2x^2$ for some $y$ such that $y^2-2x^2=-1$. The fact that there are
infinitely many such $x$ follows from the theory of Pell's equation.
The first few $s$ that arise are 0, 7, 48, 287, 1680, 9799, etc.)

\section{Preliminaries}\label{prelims}

In this section we establish our postulational criteria
for containment. We use two basic but surprisingly powerful
ideas.

\subsection{The Containment Principles}

The first idea, given homogeneous
ideals $I$ and $J$ in $k[{\bf P}^N]$, is that by examining
the zero loci of $I_t$ and $J_t$ (called $t$ degree envelopes in \cite{refTe}) 
we get a necessary criterion for
containment. In particular, if $I\subseteq J$, then the zero locus
of $I_t$ must contain the zero locus of $J_t$ in every degree $t$.
This is useful when trying to show that containment fails.

The second idea uses the obvious fact
that $I^{(m)}\subseteq I^{(r)}$ if $r\le m$, and the fact (when $I$ defines a
0-dimensional subscheme) that
$(I^{(r)})_t=(I^r)_t$ for $t$ large enough. Given $r$, 
if we pick $m\ge r$ large enough, then $\alpha(I^{(m)})$
will be large enough so that  $(I^{(r)})_t=(I^r)_t$ for all $t\ge \alpha(I^{(m)})$,
and hence $(I^{(m)})_t\subseteq (I^{(r)})_t=(I^r)_t$ for $t\ge \alpha(I^{(m)})$.
Since $(I^{(m)})_t=(0)\subseteq (I^r)_t$ for $t<\alpha(I^{(m)})$, we obtain
$I^{(m)}\subseteq I^{r}$.

Given a homogeneous ideal $J\subseteq k[{\bf P}^N]$,
let $h_J(t) = \dim J_t$ denote its Hilbert function. 
Let $P_J$ denote the Hilbert polynomial. Thus $\alpha(J)$, 
defined when $J\ne 0$,
is the least $t\ge0$ such that $h_J(t)>0$, and $\tau(J)$
is the least $t$ such that $h_J(t)=P_J(t)$.

\subsection{Some Notation for Fat Flats}

We now recall a convenient notation for denoting fat flats.
Let $I\subseteq k[{\bf P}^N]$ be any ideal of the form
$I=\cap_iI(L_i)^{m_i}$, where each $L_i\subsetneq {\bf P}^N$ is a proper linear subspace,
with no $L_i$ containing $L_j$, $j\ne i$, and where each $m_i$ is a nonnegative integer.
The fat flat subscheme $Z$ defined by $I$ depends only on
the spaces $L_i$, the integers $m_i$ and the space
${\bf P}^N$ containing $Z$. Since the latter is usually clear from context, it is 
convenient to denote the subscheme formally by $Z=m_1L_1+\cdots+m_nL_n$
and write $I=I(Z)$ for the defining ideal. In particular, 
$I(mZ)=I(Z)^{(m)}$ for each positive integer $m$.

Given a fat flat subscheme $Z=m_1L_1+\cdots+m_nL_n\subsetneq {\bf P}^N$,
the set $\hbox{Supp}(Z)=\{L_i : m_i>0\}$ is called the {\it support\/} of $Z$. In case
$Z$ is a fat point subscheme, 
we denote the sum $\sum_i \binom{m_i+N-1}{N}$ by $\hbox{deg}(Z)$;
as is well known, $P_{I(Z)}(t) = \binom{t+N}{N} - \hbox{deg}(Z)$.
It is easy to see that $\hbox{deg}(rZ)$ is a strictly increasing function of $r$.

\subsection{Preliminary Lemmas}

We begin by considering $\gamma(I)$.

\begin{lem}\label{subadd} For any homogeneous ideal $0\ne I\subseteq k[{\bf P}^N]$,
the limit 
$$\gamma(I)=\lim_{m\to\infty}\alpha(I^{(m)})/m$$ 
exists. Moreover, if $I=I(Z)$, where $Z$ is the reduced subscheme $Z\subsetneq {\bf P}^N$
consisting of a finite generic set of $n$ points, we have $\gamma(Z) = 
n(\varepsilon(N,Z))^{N-1}$.
\end{lem}

\begin{proof} This is proved in Remark III.7 of \cite{refHR2}.
For the reader's convenience we recall the proof here.

First we show $\gamma(I)$ is defined.
Note that $\alpha$ is subadditive
(i.e., $\alpha(I^{(m_1+m_2)})\le \alpha(I^{(m_1)})+\alpha(I^{(m_2)})$,
and hence $\alpha(I^{(n)})/n \le (qm/n)\alpha(I^{(m)})/m+\alpha(I^{(r)})/n\le 
\alpha(I^{(m})/m+\alpha(I^{(r)})/n$ for any positive integers
$n=mq+r$, and $\alpha(I^{(n)})/n \le \alpha(I^{(m})/m$ if $r=0$. 
Thus $\alpha(I^{(n!)})/n!\le \alpha(I^{(m)})/m$ whenever $m$ divides $n!$.
Thus $\alpha(I^{(n!)})/n!$ is a non-increasing sequence, and 
hence has some limit $c$. In addition, for all $d\ge n!$, using integer division
to write $d=q(n!)+r$ with $0\le r< n!$, we have $\alpha(I^{(d)})/d \le 
\alpha(I^{(n!)})/(n!)+\alpha(I^{(r)})/d$. It follows that the limit exists and is equal to $c$.

For the second statement, argue as in the proof of Corollary 5 of \cite{refR} 
to reduce to the case that the multiplicities are all equal. This uses
the fact that the points are generic and thus one can, essentially, average over the points.
Now we see that $n(\varepsilon(N,Z))^{N-1}$ is, by definition, the infimum
of the sequence $\alpha(I^{(n)})/n$ whose limit defines $\gamma(I)$, but it is obvious from our 
argument above that $\gamma(I)\le \alpha(I^{(n)})/n$, and hence $\gamma(I)$ is the infimum.
\end{proof}

We now give a criterion for containment to fail, and thence a lower bound for $\rho(I)$:

\begin{lem}[Postulational Criterion 1]\label{postcrit1}
Let $0\ne I\subsetneq k[{\bf P}^N]$ be a homogeneous ideal. 
Then $\gamma(I)\ge1$ and we have:
\begin{itemize}
\item[(a)] If $r\alpha(I)>\alpha(I^{(m)})$, then $I^r$ does not contain $I^{(m)}$.
\item[(b)] If $m/r< \alpha(I)/\gamma(I)$, then, for all $t\gg 0$, $I^{rt}$ does not contain $I^{(mt)}$. 
In particular, $1\le\alpha(I)/\gamma(I)\le \rho(I)$.
\end{itemize}
\end{lem}

\begin{proof} For $\gamma(I)\ge1$, see \cite[Lemma 8.2.2]{refPSC}.

(a) This is because $(I^r)_t = 0$ but $(I^{(m)})_t \ne 0$
for $t = \alpha(I^{(m)})$, since $\alpha(I^r)=r\alpha(I)>\alpha(I^{(m)})$.

(b) Suppose $m/r<\alpha(I)/\gamma(I)$. Let $0< \delta$ be such that
$m/r < \alpha(I)/(\delta+\gamma(I))$.
By definition, $\alpha(I^{(mt)})/(mt) \le \gamma(I)+\delta$ for $t\gg0$,
so $\alpha(I^{(mt)}) \le mt(\gamma(I)+\delta)<rt\alpha(I)$ for $t\gg0$,
and hence $I^{rt}$ does not contain $I^{(mt)}$ for $t\gg 0$, which now
implies $\alpha(I)/\gamma(I)\le \rho(I)$. Finally, by subadditivity, as in the proof of
Lemma \ref{subadd}, we have $\gamma(I)\le \alpha(I)$, hence
$1\le \alpha(I)/\gamma(I)$. 
\end{proof}

It is possible to give refined versions of Lemma \ref{postcrit1},
in which both $(I^r)_t$ and $(I^{(m)})_t$ may be nonzero,
but in which the zero locus of the former is bigger than that of the latter.
These refined versions are useful in doing examples and
will be the topic of a subsequent paper, \cite{refBH2}. 

We next develop our criteria for containment to hold.
First we recall a few well known facts.

\begin{lem}\label{GGPlem} Let $0\ne I\subsetneq k[{\bf P}^N]$ be a homogeneous ideal.
\begin{itemize}
\item[(a)] If $I^{(m)} \subseteq I^r$, then $r\le m$.
\item[(b)] We have $\alpha(I^{(m)})\le m\alpha(I)$ and $\alpha(I)\le {\rm reg}(I)$.
\item[(c)] If $I$ defines a 0-dimensional subscheme and $t \ge r{\rm reg}(I)$, 
then $(I^r)_t = (\hbox{sat}(I^r))_t$;
in particular, if $I$ is saturated and defines a 0-dimensional subscheme, then 
then $t \ge r\sigma(I)$ implies $(I^r)_t = (I^{(r)})_t$.
\end{itemize}
\end{lem}

\begin{proof}
(a) We have $I^m\subseteq I^{(m)} \subseteq I^r$, hence
$m\alpha(I) = \alpha(I^m) \ge \alpha(I^r) = r\alpha(I)$
so $m\ge r$.

(b) The claim $\alpha(I^{(m)})\le m\alpha(I)$ follows by the subadditivity
of $\alpha$.
The second claim is immediate from the definition of regularity,
since ${\rm reg}(I)$ is at least as much as the degree of the 
homogeneous generator of greatest degree in any minimal set of homogeneous generators of $I$,
while $\alpha(I)$ is the degree of the generator of least degree.

(c) We argue as in the proof of Proposition 2.1 of \cite{refAV}.
By Theorem 1.1 of \cite{refGGP}, $r{\rm reg}(I)\ge {\rm reg}(I^r)\ge \hbox{satdeg}(I^r)$, hence 
$t\ge r{\rm reg}(I)$ implies $(I^r)_t=(\hbox{sat}(I^r))_t$.
The second statement is just an instance of the first.
\end{proof} 

Here we give a criterion for containment to hold:

\begin{lem}[Postulational Criterion 2]\label{postcrit2}
Let $I\subseteq k[{\bf P}^N]$ be a homogeneous ideal 
(not necessarily saturated) defining
a 0-dimensional subscheme.
If $r{\rm reg}(I) \le \alpha(I^{(m)})$, then $I^{(m)} \subseteq I^r$.
\end{lem}

\begin{proof} First, $r{\rm reg}(I) \le \alpha(I^{(m)})\le m\alpha(I)\le m{\rm reg}(I)$,
so $r\le m$, hence $(I^{(m)})_t \subseteq (I^{(r)})_t$ for all $t\ge0$.
Moreover, if $I$ is not saturated, then the maximal homogeneous ideal 
$M$ is an associated prime, so $I^{(m)}=I^m$ for all
$m\ge 1$, hence $I^{(m)}=I^m\subseteq I^r$. Thus 
we may as well assume that $I$ is saturated.
But $(I^r)_t= (I^{(r)})_t$ by Lemma \ref{GGPlem}(c)
for $t \ge r{\rm reg}(I)$, 
while $r{\rm reg}(I) \le \alpha(I^{(m)})$
implies $(I^{(m)})_t=0 \subseteq (I^r)_t$ for $t<r{\rm reg}(I)$.
\end{proof}

As an application of Postulational Criterion 2 we have:

\begin{cor}\label{PCcor} Let $I\subseteq k[{\bf P}^N]$ be a homogeneous ideal 
(not necessarily saturated) defining
a 0-dimensional subscheme. If $c$ is a positive
real number such that $mc \leq \alpha(I^{(m)})$ for all $m \ge 1$,
then $I^{(m)} \subseteq I^r$ if $m/r \ge {\rm reg}(I)/c$; in particular,
$\rho(I) \le {\rm reg}(I)/c$.
\end{cor}

\begin{proof} By Lemma \ref{postcrit2},
$r{\rm reg}(I)\le mc$, or equivalently
${\rm reg}(I)/c\le m/r$, implies $I^{(m)} \subseteq I^r$.
\end{proof}

\begin{Rmk}\label{SCrem}\rm
Let $I\subseteq k[{\bf P}^N]$ be a homogeneous ideal defining a 0-dimensional subscheme.
Since we can evaluate limits on subsequences
and since by subadditivity the sequence $\alpha(I^{(i!m)})/(i!m)$ is non-increasing,
we see that $\gamma(I)\le \alpha(I^{(m)})/m$ for all $m\ge 1$.
Thus the $c$ in Corollary \ref{PCcor} can be taken to be $\gamma(I)$. 
It is reasonable to ask: why not just take
$c=\gamma(I)$? Unfortunately, the exact value of $\gamma(I)$ is rarely known
even if  $I=I(Z)$ for a fat point subscheme 
$Z=m_1p_1+\cdots+m_np_n$ in ${\bf P}^2$,
so it is useful that the statement not be in terms of $\gamma(I)$. 
On the other hand, good lower bounds are known for $\gamma(Z)$ 
in certain cases
(see for example \cite{refB}, \cite{refH1}, \cite{refHR}, \cite{refST} 
and \cite{refT}, among many others). Also, exact values are known in some cases,
such as when $\hbox{Supp}(Z)$ consists of any $n\le 8$ points
in ${\bf P}^2$. (Since the subsemigroup of classes of effective divisors for a blow up of ${\bf P}^2$ at
$n\le 8$ points is polyhedral and the postulation for any such $Z$ is known, 
one can explicitly determine $\gamma(Z)$ in this situation if one knows
the effective subsemigroup.
The effective subsemigroups for all subsets of $n\le 8$ points of the plane
are now known, as a consequence of the classification of
the configuration types of $n\le 8$ points of ${\bf P}^2$, given in
\cite{refGuardoHar} for $n\le 6$ and 
\cite{refGHM} for $7\le n\le 8$.) 
\end{Rmk}

\begin{cor}\label{Bndscor} Let $I=I(Z)$ for a nontrivial fat point subscheme $Z\subset{\bf P}^N$.  
If $\alpha(I) = \sigma(I)$, then $\rho(I) = \alpha(I)/\gamma(I)$.
\end{cor}

\begin{proof} This is immediate from Corollary \ref{PCcor},
Remark \ref{SCrem} and Lemma \ref{postcrit1}. 
\end{proof}

\begin{Rmk}\label{Omegarem}\rm
One can sometimes do better using non-postulational data.
The paper \cite{refEHU} gives various bounds
on the regularity under various assumptions.
For another example that we will refer to in Section \ref{AECQ}, 
let $I\subset k[{\bf P}^N]$ be a homogeneous
ideal defining a 0-dimensional subscheme. Then ${\rm reg}(I^r)\le r\omega(I)+
2({\rm reg}(I)-\omega(I))$ for any $r\ge 2$ by Theorem 0.4 of \cite{refCh}
(or see Section 6 of \cite{refCh2}), 
where $\omega(I)$ is the maximum degree
of a generator in any minimal set of homogeneous generators of $I$. 
Replacing $r{\rm reg}(I)$ by  $r\omega(I)+
2({\rm reg}(I)-\omega(I))$ in the argument
of the proof of Lemma \ref{postcrit2} and then arguing as in
Lemma \ref{PCcor}, keeping in mind Remark \ref{SCrem},
gives $I^{(m)} \subseteq I^r$ if 
$m/r\ge (\omega(I)+2({\rm reg}(I)-\omega(I))/r)/\gamma(I)$.
\end{Rmk}

\begin{cor}\label{asympcor} Let $I$ define a 0-dimensional subscheme of ${\bf P}^N$
and let $c>\omega(I)/\gamma(I)$.  
Then $I^{(m)} \subseteq I^r$ for all but finitely many pairs $(m,r)$ with $m/r\ge c$.
In particular, if $m/r > \omega(I)/\gamma(I)$, then $I^{(mt)} \subseteq I^{rt}$
for all $t\gg 0$.
\end{cor}

\begin{proof} By Remark \ref{Omegarem}, we have $I^{(m)} \subseteq I^r$
if $(m,r)$ is on or above the line 
$m = (\omega(I)/\gamma(I))r+2({\rm reg}(I)-\omega(I))/\gamma(I)$.
But $c$ is greater than the slope $\omega(I)/\gamma(I)$ of this line, so there are only
finitely many pairs $(m,r)$ with $m/r \ge c$ below this line.
The second statement is now immediate.
\end{proof}

\subsection{Constructions showing Optimality}\label{Optsubsect}

To prove Corollary \ref{LESHHmaxThm}, it suffices to find
subschemes $Z\subseteq {\bf P}^N$ for which 
$\rho(Z)$ is large. Lemma \ref{postcrit1} suggests where to look.
We want a scheme $Z$ such that 
$\alpha(I(Z))$ is as large as possible, which means that $I(Z)$
should behave generically, from a postulational point of view.
On the other hand, we want $\gamma(Z)$ to be small, so among all $I(Z)$
with generic Hilbert function we want to examine those
for which the Hilbert function of $I(Z)^{(m)}$ is as large as possible
(and hence $\alpha(I(Z)^{(m)})$ is as small as possible).

This problem was studied in \cite{refGMS} in characteristic 0
in the case that $N=m=2$ with $Z=p_1+\cdots+p_n$ 
a reduced set of points $p_i$; i.e., double points in the plane.
They prove that the the set of singular points of a union of
$s$ general lines (i.e., the pair-wise intersections of $s$ general
lines) is a configuration of points in the plane
having generic Hilbert function but for which 
the Hilbert function of the symbolic square of the ideal is as 
large as possible. This suggests, more generally, to look at the set of 
$N$-wise intersections of $s\ge N+1$ general hyperplanes in ${\bf P}^N$.
More generally yet, for $1\le e\le N$ and $s\ge e$, 
let $S_N(e, s, {\bf d})$ denote the
reduced scheme consisting of the $e$-wise intersections of
$s$ general hypersurfaces $H_1,\ldots,H_s$ in ${\bf P}^N$
of respective degrees $d_i$ where ${\bf d} = (d_1,\ldots,d_s)$. 
If $d_i=d$ for all $i$, we will write $S_N(e, s, d)$ for $S_N(e, s, {\bf d})$.
If $d=1$, we will write simply $S_N(e, s)$.
Thus $S_N(N,N+1)$ can be taken to be 
the set of coordinate vertices of ${\bf P}^N$,
and $S_N(1,N+1)$ to be the union of the coordinate hyperplanes.
In this notation, the examples of Ein having large $\rho$
are the codimension $e$ skeleta $S_N(e,N+1)$ 
of the coordinate simplex
in ${\bf P}^N$ (hence $d_i=1$ for all $i$); i.e., 
the $e$-wise intersections of $s=N+1$ general hyperplanes
in ${\bf P}^N$. The case $e=N$ (i.e., of the coordinate vertices
in ${\bf P}^N$) is treated by Arsie and Vatne (see Theorem 4.5
of \cite{refAV}).

It is easy to see that a general hyperplane section 
$H\cap S_N(e, s,{\bf d})$ is $S_{N-1}(e, s, {\bf d})$, defined
by the $e$-wise intersections of the hypersurfaces $H\cap H_i\subseteq H$.
We will denote $\alpha(I(mS_N(e,s,{\bf d})))$  by $\alpha_N(m,e,s,{\bf d})$, where
$mS_N(e,s,{\bf d})\subseteq {\bf P}^N$ is the subscheme consisting of the $e$-wise intersections of the 
$s$ hypersurfaces $H_i$, where each $e$-wise intersection is taken with multiplicity $m$.

In order to apply our bounds to $S_N(e, s,{\bf d})$, we need to determine
the least degree among hypersurfaces that vanish on $mS_N(e, s, {\bf d})$.

\begin{lem}\label{alphaLemma} Let $1\le e\le N$, $s\ge e$,
and let ${\bf d}=(d_1,d_2,\ldots,d_s)$. Let 
$I=I(mS_N(e, s, {\bf d}))\subseteq k[{\bf P}^N]$.
If $m=re$ for some $r$, then $r(d_1+\cdots+d_s)\ge \alpha(I)$.
If $d_1=\cdots=d_s=1$, then for any $m\ge1$ we have 
$ms/e\le \alpha(I)$, and hence we have equality if 
$m=re$.
\end{lem}

\begin{proof}
First consider the case that $m=re$ is a multiple of $e$.
Then the divisor $r(H_1+\cdots+H_s)$ has degree 
$r(d_1+\cdots+d_s)=m(d_1+\cdots+d_s)/e$ and vanishes 
on each component of $S_N(e,s, {\bf d})$ with multiplicity 
$m$ (since each component of $S_N(e,s, {\bf d})$ is contained
in exactly $e$ of the hypersurfaces $H_i$). Thus 
$r(d_1+\cdots+d_s)\ge \alpha_N(m,e,s,{\bf d})$. 

Now assume $d_1=\cdots=d_s=1$.
To show $\alpha_N(m,e,s,{\bf d})\ge ms/e$, it is enough to show
$\alpha_e(m,e,s,{\bf d})\ge ms/e$, since by taking 
general hyperplane sections we have:\newline
$\alpha_N(m,e,s,{\bf d})\ge \alpha_{N-1}(m,e,s,{\bf d})
\ge \cdots \ge \alpha_{e}(m,e,s,{\bf d})$.

Suppose it were true that
$\alpha_{e}(m,e,s,{\bf d})< ms/e$ for some $m$. Let $F$ be a form of 
degree $d=\alpha_{e}(m,e,s,{\bf d})$
vanishing with multiplicity at least $m$ at each point of 
$S_e(e,s, {\bf d})$.
Then $F$ restricts to give a form on $H_1$
with $d < ms/e \le m(s-1)/(e-1)$, but $H_1\cap S_{e}(e,s,{\bf d})
=S_{e-1}(e-1,s-1,{\bf d}')$, where ${\bf d}'=(d_2,\cdots,d_s)$,
and, by induction on the dimension (where dimension 1 is easy),
we have $m(s-1)/(e-1)\le \alpha_{e-1}(m,e-1,s-1,{\bf d}')$.
Hence $F$ vanishes identically on $H_1$. By symmetry,
$F$ vanishes on all of the hyperplanes $H_i$. Dividing out by the
linear forms defining the hyperplanes gives a form $F'$ of degree
$d-s$ vanishing with multiplicity $m-e$ at each point of 
$S_{e}(e,s,{\bf d})$, and hence 
$\alpha_{e}(m-e,e,s,{\bf d})\le d-s < ms/e-s = (m-e)s/e$,
hence again $F'$ vanishes on all $H_i$. Continuing in this way,
we eventually obtain a form of degree less than $s$ 
that vanishes on the $s$ hyperplanes $H_i$, which is a contradiction
unless $F=0$.
\end{proof}

We still need to know $\alpha(I(S_N(e, s, {\bf d}))$.

\begin{lem}\label{regLem} Let $1\le e\le N$, $e\le s$
and $d_1\le d_2\le \cdots\le d_s$.
For $S=S_N(e,s,{\bf d})$ we have $\alpha(I(S)) = d_1+\cdots+d_{s-e+1}$.
If $e=N$ and $d_i=1$ for all $i$, we have 
$\alpha(I(S)) =\sigma(I(S)) = s-N+1$.
\end{lem}

\begin{proof} Clearly, $\alpha(I(S)) \le d_1+\cdots+d_{s-e+1}$,
since every intersection of $e$ of the hypersurfaces must involve
one of the hypersurfaces $H_1,\ldots,H_{s-e+1}$.
For the rest, let us refer to the union of the $e$-wise intersections
of the hypersurfaces $H_i$ as the codimension $e$ skeleton of the $H_i$,
or just the $e$-skeleton. We will now show that
any hypersurface $H$ of degree $d<d_1+\cdots+d_{s-e+1}$ which vanishes
on the $e$-skeleton also vanishes on the $(e-1)$-skeleton.
Since $d<d_1+\cdots+d_{s-e+1}\le d_1+\cdots+d_{s-(e-1)+1}$, 
this means that $H$ also vanishes on the $(e-1)$-skeleton, and so on,
and thus vanishes on the 1-skeleton and indeed the 0-skeleton 
(i.e., the whole space, since a form of degree $d$
cannot contain hypersurfaces whose degrees sum to more than $d$). 
Thus $H$ is 0, and this shows 
$\alpha(I(S)) \ge d_1+\cdots+d_{s-e+1}$ which gives equality.

So suppose $H$ has degree $d<d_1+\cdots+d_{s-e+1}$ and vanishes
on the $e$-skeleton. Thus for any indices $i_1<\cdots<i_{e-1}$ and any $j$
not one of these indices,
$H$ vanishes on $H_{i_1}\cap \cdots\cap H_{i_{e-1}}\cap H_j$.
By Bertini (Theorem II.8.18 of \cite{refHt}, taking hyperplane sections
after uple embeddings), intersections of general hypersurfaces
are smooth and, in dimension 2 or more, irreducible. Thus 
$H_{i_1}\cap \cdots\cap H_{i_{e-1}}$ is irreducible.
If it were not already contained in $H$, we can
intersect with $H$ and do a degree calculation:
$H\cap H_{i_1}\cap \cdots\cap H_{i_{e-1}}$ has degree $dd_{i_1}\cdots d_{i_{e-1}}$
whereas the union of the intersections of $H_{i_1}\cap \cdots\cap H_{i_{e-1}}$ with all
possible $H_j$ (i.e., for all $j$ not among the indices 
$i_1,\ldots,i_{e-1}$), has degree $d_{i_1}\cdots d_{i_{e-1}}\sum_jd_j$, where the sum
is over all $j$ not among $i_1,\ldots,i_{e-1}$.
Clearly $d<d_1+\cdots+d_{s-e+1}\le \sum_jd_j$ since the $d_i$ are assumed to be
nondecreasing. Since the total degree of the intersection
$H\cap H_{i_1}\cap \cdots\cap H_{i_{e-1}}$ is less than the sum of the degrees of the 
divisors contained in the intersection, it follows that 
$H_{i_1}\cap \cdots\cap H_{i_{e-1}}\subseteq H$ for each component
of the $(e-1)$-skeleton, as claimed.

Finally, suppose $e=N$ and $d_i=1$ for all $i$. Then as we have just seen,
$\alpha(S)=s-e+1$. But there are $\binom{s}{e}$ points and 
$\binom{(s-e)+N}{N}=\binom{s}{e}$ forms of degree $s-e$ in $N+1$ variables.
Thus the number of conditions imposed by the points equals the number
of points, hence $\tau(I)=s-e$ so $\sigma(I)=\alpha(I)=s-N+1$.
\end{proof}

We now can obtain some results on $\rho(S_N(e,s))$.
As noted above, Theorem \ref{skeletonThm}(b) in the case $s=N+1$
is due to L. Ein; Theorem 4.5 of \cite{refAV} implies
$2-1/N\le \rho(S_N(N,N+1))$, using as Ein 
did the fact that the ideal is monomial.

\begin{thm}\label{skeletonThm}
Let $1\le e\le N$ and $e\le s$. Then:
\begin{itemize}
\item[(a)] $\rho(S_N(N,s))=N(s-N+1)/s$; and
\item[(b)] $e(s-e+1)/s\le \rho(S_N(e,s))$.
\item[(c)] More generally, given ${\bf d}=(d_1,\ldots,d_s)$
with $d_1\le \cdots\le d_s$,
we have $$e(d_1+\cdots+d_{s-e+1})/(d_1+\cdots+d_s)\le \rho(S_N(e,s,{\bf d})).$$
\end{itemize}
\end{thm}

\begin{proof} By Lemma \ref{regLem},
$\alpha(I(S_N(e,s)))=s-e+1$, $\sigma(I(S_N(N,s)))=s-N+1$,
and $\alpha(I(S_N(e,s,{\bf d})))=d_1+\cdots+d_{s-e+1}$,
while by Lemma \ref{alphaLemma}, we see that
$$\gamma(I(S_N(e,s)))=\lim_{m\to\infty}\frac{\alpha(I(meS_N(e,s)))}{(me)}=s/e$$
and similarly $\gamma(I(meS_N(e,s,{\bf d})))\le (d_1+\cdots+d_s)/e$.

(a) By Corollary \ref{Bndscor}, we thus have $\rho(S_N(N,s))=N(s-N+1)/s$.

(b) By Lemma \ref{postcrit1} we have $e(s-e+1)/s\le \rho(S_N(e,s))$.

(c) By Lemma \ref{postcrit1} we have $e(d_1+\cdots+d_{s-e+1})/(d_1+\cdots+d_s)
\le \rho(S_N(e,s,{\bf d}))$.
\end{proof}

\subsection{General Facts about $\rho$}

Here we take note of some general behavior of $\rho(I)$.
To state the results, let $R=k[{\bf P}^N]$, let 
$x$ be an indeterminate with respect to which we have
$R\subseteq R[x]=k[{\bf P}^{N+1}]$, and let the quotient
$q: R[x]\to R$ correspond to the inclusion ${\bf P}^N\subseteq {\bf P}^{N+1}$.
If $I\subseteq R$ is a homogeneous ideal, let $I'=IR[x]$ be the extended ideal.
In case $I=I(Z)$ for some subscheme $Z\subseteq {\bf P}^N$, we will denote by $C(Z)$
the subscheme defined by $I'$; we note that $C(Z)$ is 
just the projective cone over $Z$. 

\begin{prop}\label{coneprop} In the notation of the preceeding paragraph, we have:
\begin{itemize}
\item[(a)] $\rho(I)=\rho(I')$, hence $\rho(Z)=\rho(C(Z))$ for any 
nontrivial subscheme $Z\subsetneq {\bf P}^N$;
\item[(b)] $\rho(I) =\rho(q^{-1}(I))$, hence if $I=I(Z)$ for a nontrivial subscheme $Z\subsetneq {\bf P}^N\subseteq {\bf P}^{N+1}$,
then $\rho(Z)$ is well defined, whether we regard $Z$ as being in ${\bf P}^N$ or ${\bf P}^{N+1}$; and
\item[(c)] $\rho(mZ)\le \rho(Z)$ for any fat flat subscheme $Z$.
\end{itemize}
\end{prop}

\begin{proof} (a) Since $R\to R[x]$ is flat, primary decompositions
of ideals in $R$ extend to primary decompositions in $R[x]$ (see \cite{refMa}, Theorem 13,
or Exercise 7, \cite{refAM}). 
Since $I$ and $I'$ have the same generators,
whenever $I$ and $J$ are ideals in $R$, we have $I\subseteq J$ if and only if
$I'\subseteq J'$. Taken together, this means $I^{(m)}\subseteq I^r$ if and only if
$(I')^{(m)}\subseteq (I')^r$, and hence that $\rho(I)=\rho(I')$.

(b) Note that $q^{-1}(I)=I'+(x)$, and use the facts that $(q^{-1}(I))^r=\sum_i(x^i)(I')^{r-i}$
and $(q^{-1}(I))^{(m)}=\sum_j(x^j)(I')^{(m-j)}$. If $(q^{-1}(I))^{(m)}\subseteq (q^{-1}(I))^{r}$,
setting $x=0$ gives $I^{(m)}\subseteq I^r$, and hence $\rho(q^{-1}(I))\ge \rho(I)$.
And if $m/r>\rho(I)=\rho(I')$, then $(m-j)/(r-j)\ge m/r$ for $0\le j<r$, so
$x^j(I')^{(m-j)}\subseteq x^j(I')^{r-j}$ hence $(q^{-1}(I))^{(m)}\subseteq (q^{-1}(I))^r$,
so $\rho(I)\ge \rho(q^{-1}(I))$.

(c) By definition we can find a ratio $s/r <  \rho(I^{(m)})$
arbitrarily close to $\rho(I^{(m)})$ such that $(I^{(m)})^r$ does not contain
$I^{(sm)}$, hence $I^{rm}$ does not contain
$I^{(sm)}$, so $sm/(sr) <  \rho(I)$.
\end{proof}

Equality in Proposition \ref{coneprop}(c) can fail. 
For example, if $Z$ is the reduced union of three general points in ${\bf P}^2$,
then $\rho(mZ)= 1$ if $m$ is even, while $\rho(mZ)=(3m+1)/(3m)$ if $m$ is odd
\cite{refBH2}.

\section{Proofs}\label{Proofs}

\begin{proof}[Proof of Corollary \ref{LESHHmaxThm}]
The result of \cite{refHH1} shows that
$\rho(N,e)\le e$, while taking the limit as $s\to\infty$
in Theorem \ref{skeletonThm}(b) shows $e\le \rho(N,e)$.
Alternatively, using Theorem \ref{skeletonThm}(a),
$i$ applications of Proposition \ref{coneprop}(a),
and then taking the limit for $s\to \infty$, we conclude
$\rho(N+i,N)=N$ for all $N$ and $i$.
\end{proof}

\begin{proof}[Proof of Theorem \ref{SCthm}]
The upper bound is an immediate
consequence of Corollary \ref{PCcor} and Remark \ref{SCrem}. The lower bound is Lemma \ref{postcrit1}.
\end{proof}

\begin{proof}[Proof of Corollary \ref{SCcor}]
Since the points are generic and the number of points is the binomial coefficient
$n=\binom{s+N}{N}$, we know $\alpha(I)=\sigma(I)=s+1$.
Also, by Lemma \ref{subadd} we know $\gamma(I)=n(\varepsilon(N,Z))^{N-1}$.
Thus the result follows immediately from Theorem \ref{SCthm}.
When $N=2$ and $n$ is a square, we know in addition that $n\varepsilon(2,n)=\sqrt{n}$.
The fact that we can also write $\frac{s+1}{\sqrt{n}}=\sqrt{2}\sqrt{\frac{s+1}{s+2}}$
follows from $n=\frac{(s+1)(s+2)}{2}$. 
\end{proof}

\renewcommand{\thethm}{\thesection.\arabic{thm}}
\setcounter{thm}{0}

\section{Additional Examples, Comments and Questions}\label{AECQ}

The inspiration for this paper was a question Huneke asked
the second author: if $S$ is a finite set of points in ${\bf P}^2$
with $I=I(S)$, is it true that $I^{(3)}\subseteq I^2$?

We cannot yet answer this question, but 
we can show Huneke's question
has an affirmative answer in many cases.
(Theorems 3.3 and 4.3 of \cite{refTY} give additional cases in which
Huneke's question has an affirmative answer.)

\begin{thm}\label{8}
Let $I=I(S)$, where $S$ is a set of n generic points of ${\bf P}^2$.
Then $I^2$ contains $I^{(3)}$ for every $n \geq 1$.
\end{thm}

\begin{proof} Since for $n=1$, $2$, or $4$, $S$ is a complete intersection
and hence $I^3 = I^{(3)}$, the theorem is true in those cases,
so assume $n$ is not $1$, $2$ or $4$.

If $2\sigma(I) \leq \alpha(I^{(3)})$, then $I^2$ contains $I^{(3)}$
by Lemma \ref{postcrit2}.
Since the points are generic and of multiplicity $1$,
they impose independent conditions in degrees $\alpha(I)$ or more, so $\sigma(I)$ is the
largest $t$ such that $\binom{t}{2} < n$.
Also, the Hilbert functions of ideals of fat point
subschemes supported at 9 or fewer generic points
are known (see, e.g., \cite{refH2}) so we can compute
$\alpha(I^{(3)})$ exactly. Here's what happens for $n \leq 9$:
\[
\begin{array}{c|c|c}
n	& \sigma(I) &	\alpha(I^{(3)})\\
\hline 
3 & 2 &  5\\
5 & 3 & 6 \\ 
6 & 3 & 8\\
7 & 4 & 8\\
8 & 4 & 9\\
9 & 4 & 9
\end{array}
\]

We see that $2\sigma(I) \leq \alpha(I^{(3)})$, hence $I^2$ contains $I^{(3)}$.

Now assume $n \geq 10$. Let $t = \sigma(I)$; then $\binom{t}{2} < n$
so $t^2-t < 2n$. (Also note that since $n \geq 10$, we must have $t \geq 4$.)
It is known (see \cite{refHi} or \cite{refM}) that $n\ge 10$ points of multiplicity $3$
impose independent conditions on forms of degree at least $\alpha(I^{(3)})$, 
so $\alpha(I^{(3)})$ is the least $d$ such that $\binom{d+2}{2} > 6n$.
Thus to show $2t \leq d$, it is enough to show that $\binom{2t+1}{2} \leq 6n$,
which we will do using the fact that $t^2-t < 2n$ and hence
$3t^2-3t < 6n$. In fact, all we need do now is verify that
$\binom{2t+1}{2} \leq 3t^2-3t$, which is easy, keeping in mind that
$t \geq 4$.
\end{proof}

We can also give an affirmative answer to a stronger
version of Huneke's question in the case of $n$ generic points.

\begin{thm}\label{genThm} Let $I=I(S_n)$, where $S_n$ is a 
set of n generic points of ${\bf P}^2$.
Then $I^r$ contains $I^{(m)}$ whenever $m/r > 3/2$.
\end{thm}

\begin{proof} This amounts to showing that $\rho(S_n) \leq 3/2$.
Again we can ignore $n = 1$, $2$ and $4$. For $n= 3$, $5$, $6$, $8$ and $9$ 
we use the upper bounds $\sigma(I)/(n\varepsilon(2,n))$ for $\rho(S_n)$ 
obtained from Corollary \ref{PCcor} with $c=n\varepsilon(2,n)$; 
the values of $n\varepsilon(2,n)$ can be obtained
from Nagata's list of abnormal curves \cite{refNag}, or see
\cite{refH4}. By Corollary \ref{Bndscor},
$\sigma(I)/(n\varepsilon(2,n))=\rho(S_n)$ when $n=3$ and $n=6$. (Using refined methods
that we will present in a subsequent paper, \cite{refBH2}, 
we can show that equality holds also for $n=8$ and 9, and that $\rho(S_5)=6/5$
and $\rho(S_7)=8/7$.) 
\[
\begin{array}{c|c|c|c}
n & n\varepsilon(2,n) & \sigma(I) & \frac{\sigma(I)}{n\varepsilon(2,n)} \\
\hline 
3 &  \frac{3}{2} & 2 & \frac{4}{3} \\
& & & \\
5 & 2 & 3 & \frac{3}{2}\\
& & & \\
6 & \frac{12}{5} & 3 & \frac{5}{4}\\
& & & \\
7 & \frac{21}{8} & 4 & \frac{32}{21} = 1.52\\
& & & \\
8 & \frac{48}{17} & 4 & \frac{17}{12} = 1.42\\
& & & \\
9 & 3 & 4 & \frac{4}{3}	
\end{array}
\]

In order to handle $n = 7$, we see we need a better bound,
which we obtain using Remark \ref{Omegarem}.
We claim $\rho(S_7)\le 6/5$. We must show that
if $m/r > 6/5$, then $\alpha(I^{(m)})\ge r\omega(I)+2({\rm reg}(I)-\omega(I))$,
where here $\omega(I)=3$ and ${\rm reg}(I)=4$
(see \cite{refH3} for the graded Betti numbers for the resolution of the ideal $I^{(m)}$
for any $m\ge 1$). From the Seshadri constant in the table above,
we know $\alpha(I^{(m)})\ge 21m/8$. (In fact, it turns out that
$\alpha(I^{(m)})=\lceil 21m/8\rceil$. Clearly $\alpha(I^{(m)})\ge \lceil 21m/8\rceil$,
and one checks the cases $m\le 8$ directly to see that equality holds.
For $m\ge 8$, write $m=8i+j$ with $0\le j<8$ and
use $\alpha(I^{(8i+j)})\le i\alpha(I^{(8)})+\alpha(I^{(j)})=\lceil 21m/8\rceil$.)
Thus $\lceil 21m/8\rceil\ge 3r+2$ (or, equivalently, $21m/8>3r+1$) implies
$\alpha(I^{(m)})\ge r\omega(I)+2({\rm reg}(I)-\omega(I))$.
But for $r>6$, $m/r\ge 6/5$ implies $21m/8>3r+1$ and hence
$\alpha(I^{(m)})\ge r\omega(I)+2({\rm reg}(I)-\omega(I))$.
We now check $r\le 6$ individually. If $r=1$, clearly for any $m\ge1$ we have
$I^{(m)}\subseteq I^r$. For $r=2$ and $m/r\ge 1.2$, we have $m\ge3$,
$\alpha(I^{(m)})\ge8$, and so $\alpha(I^{(m)})\ge8=3r+2$. 
Similarly for $r=3,4,5$ and 6. Thus $\rho(S_7)\le 1.2$.
(We cannot do better than $\rho(S_7)\le 1.2$ using this argument, since $m=6$ and $r=5$
give $m/r=1.2$, yet fail to satisfy $\alpha(I^{(m)})\ge 3r+2$.)

Now consider $n > 9$. It is known that $\varepsilon(n) \ge \sqrt{n-1}/n$.
See [Xu] for characteristic 0. It also follows from \cite{refH1}
in all characteristics, as follows. Let $s = \lfloor\sqrt{n}\rfloor$,
and define $0\le t \le s$ so that either $n = s^2+2t $ or $n = s^2 + 2t + 1$.
Let $d=s$ and $r=s^2 + t$.
First consider the case that $n = s^2+2t$.
Since $r/d \geq \sqrt{n}$, then $\varepsilon(n) \geq d/r$ by 
\cite{refH1}, and a little arithmetic shows that $d/r \geq \sqrt{n-1}/n$.
Now let $n = s^2+2t+1$.
Since now $r/d \leq  \sqrt{n}$ (keep in mind that $t < s$),
then $\varepsilon(n) \geq r/(nd)$ by \cite{refH1}, 
and it is easy to see that $r/(nd) \ge \sqrt{n-1}/n$.
So for $n > 9$ it is enough to check that
$\sigma(I)/(\sqrt{n-1}) \leq 3/2$.

Now, $\sigma(I) = t+1$ for the least t such that $\binom{t+2}{2} \geq n$.
Since for $t=(\sqrt{8n+1}-3)/2$ we have $\binom{t+2}{2} = n$, we see
$\sigma(I) \leq (\sqrt{8n+1}-3)/2 + 2$. It is not hard to check that
$((\sqrt{8n+1}-3)/2 + 2)/(\sqrt{n-1}) < 3/2$ for all $n \ge 52$.

We have left to deal with $10 \leq n \leq 51$. For these few cases we can
use the best lower bounds for $\varepsilon(n)$ given in \cite{refH1} 
(or the exact value if $n$ is a square) instead of $\sqrt{n-1}/n$,
and we can use the exact value of $\sigma(I)$
instead of $(\sqrt{8n+1}-3)/2 + 2$. Doing so we find that
$\sigma(I)/(n\varepsilon(n)) < 3/2$ except for $n = 11$ (in this case
even taking the conjectural value $\varepsilon(n) = 1/\sqrt{n}$ gives only
that $\sigma(I)/(n\varepsilon(n)) \leq 1.507)$, or $n = 17$, $22$ or $37$, in which case we
have $\varepsilon(n)$ being at least $4/17$, $7/33$ and $6/37$ resp., hence at least we
obtain $\sigma(I)/(n\varepsilon(n)) \leq 3/2$, but this suffices for the statement
of the theorem (however, see the remark that follows). 

For $n=11$, argue as for $n=7$. For $n=11$, we have $\omega(I)=4$,
and ${\rm reg}(I)=5$, so $r\omega(I)+2({\rm reg}(I)-\omega(I))=4r+2$.
Now $\rho(S_{11})\le c$ if we pick $c$ such that 
$m/r\ge c$ implies $\alpha(I^{(m)})\ge 4r+2$. But
$n\varepsilon(n)\ge\sqrt{10}$ so $\alpha(I^{(m)})\ge m\sqrt{10}\ge rc\sqrt{10}$, 
so we just need $c$ such that $rc\sqrt{10}>4r+1$ for $r\ge2$.
We see we need $c>4/\sqrt{10} + 1/(2\sqrt{10})=
(2{\rm reg}(I)-1)/(2\sqrt{n-1})$ so $c=1.43$ suffices;
i.e.,  $\rho(S_{11})\le 1.43$.
\end{proof}

\begin{Rmk} In a subsequent paper, \cite{refBH2}, we will compute $\rho(S)$ for sets of
points on irreducible plane conics. Our result for the case of 5 points on a smooth conic
is $\rho(S_5)=6/5$. Also, arguing in the case of $n=17, 22$ and 37
generic points as we did for $n=11$, we find, resp., that
$\omega(I)$ and ${\rm reg}(I)$ are 5, 6 and 8, and 6, 7 and 9, and hence that
$(2{\rm reg}(I)-1)/(2\sqrt{n-1})$ is 1.375, 1.418, 1.416, resp., so
$\rho(I)$ is, for example, at most 1.38, 1.42, and 1.42, resp.
Thus in fact we can state a slightly stronger version of
the preceding theorem:
for a generic set $S_n$ of $n$ points of ${\bf P}^2$,
$I^r$ contains $I^{(m)}$ whenever $m/r \geq 3/2$ (rather than just $m/r > 3/2$).
(Alternatively, assuming characteristic 0, we can 
handle the cases $n=17, 22$ and 37 simply by using a better
estimate for $\varepsilon(n)$: in characteristic 0, \cite{refB}
shows $\varepsilon(n)$ is at least $8/33$, $42/197$ and $12/73$, resp.)
\end{Rmk}

In fact, it may be possible that $\rho(S) \leq \sqrt{2}$
whenever $S$ is a generic finite set of points in ${\bf P}^2$.
[While this paper was under review we found that $\rho(S_8)=17/12>\sqrt{2}$ \cite{refBH2},
but we know no other cases for which $\rho(S) > \sqrt{2}$.]
In addition to Theorem \ref{SCcor}, the following result gives some evidence for this
possibility.

\begin{prop}\label{genProp} Let $S$ be a set of $n = (d+2)(d+1)/2 + i\ge10$ generic points
of ${\bf P}^2$ and $(d+4)/2 \le i \le d+2$. Then
$m/r \ge \sqrt{2}$ implies $I^{(m)} \subseteq I^r$.
\end{prop}

\begin{proof} By Lemma \ref{postcrit2} (Postulational Criterion 2),
$I^{(m)} \subseteq I^r$ if $r\sigma(I) \le \alpha(I^{(m)})$, and hence if
$m/r \ge (\sigma(I))/(nc)$, where $\varepsilon(S) \ge c$.
But here $\sigma(I)=d+2$ (since by our choice of $i$ we have
$\binom{d+2}{2}<n\le \binom{d+3}{2}$) and, as in the proof of Theorem \ref{genThm},
we can take $c=\sqrt{n-1}/n$ since $n\ge10$. A little arithmetic using $(d+4)/2 \le i$ now shows that
$\sqrt{2}\ge (\sigma(I))/(nc)$.
\end{proof}

\begin{Ex} By the main theorems of \cite{refELS} and \cite{refHH1}, 
$I^{(4)}\subseteq I^2$ for $I=I(S)$ for any finite subset 
$S\subseteq {\bf P}^2$. Thus, in addition to asking, as Huneke did, if 
$I^{(3)}\subseteq I^2$, one might also ask if $I^{(4)}\subseteq I^3$ or 
if $I^{(6)}\subseteq I^4$. We close by showing that the answer for the latter 
two is no.

In particular, let $I=I(S)$ where $S=S_2(2,s)$ is the set of $n=\binom{s}{2}$ points of 
pairwise intersection of $s$ general lines
in ${\bf P}^2$. It is easy to check that
$\alpha(I^{(3)})=2s-1$, since any form in $I^{(3)}$ of degree $2s-2$
must, by Bezout, vanish on each of the $s$ lines, giving a form
of degree $s-2$ in $I$, but $\alpha(I)=s-1$, either by Bezout again
or by Lemma \ref{alphaLemma}. (Similarly, it follows
that $\alpha(I^{(m)})=((m+1)/2)s-1$ whenever $m$ is odd.)
Now by Lemma \ref{postcrit2}, using Lemma \ref{regLem},
it follows that $I^{(3)}\subseteq I^2$ for all $s$, and 
by Lemma \ref{postcrit1}, using Lemma \ref{alphaLemma},
it follows that $I^3$ does not contain $I^{(4)}$ for $s>3$
and that $I^4$ does not contain $I^{(6)}$ for $s>4$.
\end{Ex}

\end{document}